\documentclass[11pt]{article}
\usepackage{amsfonts}\usepackage{amssymb}\usepackage{amsmath}
\usepackage{setspace,times,a4wide}
\usepackage{amssymb,graphicx,graphics}
\usepackage[utf8]{inputenc}
\usepackage{amsfonts}
\usepackage{amsmath}
\usepackage{amsthm}
\usepackage{mathtools}
\usepackage{graphicx}
\usepackage{stackengine}
\usepackage{tabularx}
\usepackage{booktabs}
\usepackage{xspace}
\usepackage{multirow}
\usepackage{cases}
\usepackage{subcaption}
\usepackage[inline]{enumitem}
\usepackage{xcolor}
\usepackage{microtype}
\usepackage{hyperref}

\usepackage{tikz}
\usetikzlibrary{calc,positioning,decorations.pathreplacing,fit}

\def\ve#1{\mathchoice{\mbox{\boldmath$\displaystyle\bf#1$}}
{\mbox{\boldmath$\textstyle\bf#1$}}
{\mbox{\boldmath$\scriptstyle\bf#1$}}
{\mbox{\boldmath$\scriptscriptstyle\bf#1$}}}

\newcommand\veb{{\ve b}}

\newcommand\vece{{\ve e}}

\newcommand\vel{{\ve l}}

\newcommand\veu{{\ve u}}

\newcommand\vew{{\ve w}}
\newcommand\vex{{\ve x}}
\newcommand\vey{{\ve y}}

\newcommand{\df}{\mathrel{\mathop:}=}

\makeatletter
\newtheorem*{rep@theorem}{\rep@title}
\newcommand{\newreptheorem}[2]{%
\newenvironment{rep#1}[1]{%
 \def\rep@title{#2 \ref{##1}}%
 \begin{rep@theorem}}%
 {\end{rep@theorem}}}
\makeatother

\newcommand{\lv}[1]{}

\newtheorem{theorem}{Theorem}

\newtheorem{definition}{Definition}
\newtheorem{proposition}{Proposition}

\newreptheorem{theorem}{Theorem}
\newreptheorem{lemma}{Lemma}

\theoremstyle{remark}

\newtheorem*{claim*}{Claim}
\newtheorem*{example*}{Example}

\newcommand{\ignore}[1]{}

\def\N{\mathbb{N}}
\def\R{\mathbb{R}}
\def\Z{\mathbb{Z}}

\def \DD {\mathcal{D}}

\DeclareMathOperator{\gap}{\mathrm{gap}}

\begin{document}

\title{Better and Simpler Reducibility Bounds over the Integers}

\author{Asaf Levin
\thanks{Faculty of Data and Decision Sciences, Technion -- Israel Institute of Technology, Haifa, Israel. Email: levinas@technion.ac.il}}

	\maketitle

\begin{abstract}
We study the settings where we are given a function of  $n$ variables defined in a given box of integers.  We show that in many cases we can replace the given objective function by a new function with a much smaller domain. Our approach allows us to transform a family of weakly polynomial time algorithms into strongly polynomial time algorithms. 
\end{abstract}
	
{\bf Keywords:} integer programming; analysis of algorithms; reducibility bounds.
	
	\section{Introduction}
Our study concerns a generalization of the following question: given a linear function $\vew \vex$, what is the smallest value of  $\|\bar{\vew}\|_\infty$ of  an equivalent function $\bar{\vew} \vex$?
	We follow previous work and say that $\bar{\vew}$ is \emph{equivalent} if $\vew \vex > \vew \vey$ implies that $ \bar{\vew} \vex > \bar{\vew}\vey$ for all feasible integral vectors $\vex, \vey$.
The seminal work of Frank and Tardos~\cite{FT} provides an algorithm to compute such a vector.  Their approach is based on the simultaneous diophantine approximation, and thus the time complexity of applying this procedure is very slow although polynomially bounded.  Since this work was published, it was used over and over again to transform a weakly polynomial time algorithm for a given subclass of integer program into a strongly polynomial time algorithm for the same subclass.  More recently, in \cite{Eisen23} it was demonstrated that if we are willing to settle on existence results, then better (i.e., smaller) bounds can be derived. This existence result is usually constructive and thus there exists an exponential time algorithm to derive the equivalent function. We follow these two research directions regarding the study of existence and efficiently computable results and further improve further these bounds.  Furthermore, we also consider new generalizations of this problem and obtain the first known bounds to these generalizations.

\paragraph{Notation.} We use the following standard notation for vectors by using boldface (e.g., $\vex$) and for their entries we use normal font (for example, the $i$-th entry of~$\vex$ is~$x_i$).
For positive integers $m \leq n$ we set $[m,n] \df \{m,\ldots, n\}$ and $[n] \df [1,n]$. For vectors we use similar notation so for $\vel, \veu \in \Z^n$ with $\vel \leq \veu$, $[\vel, \veu] \df \{\vex \in \Z^n \mid \vel \leq \vex \leq \veu\}$.

\paragraph{The model.}
For a function $f: \Z^n \to \Z$ and two vectors $\vel, \veu \in \Z^n$, we let the gap of $f$ be $$f_{\gap}^{[\vel, \veu]} \df \max_{\vex, \vex' \in [\vel, \veu]} |f(\vex) - f(\vex')| ,$$ and if $[\vel, \veu]$ is clear from the context we write $f_{\gap}$.

We define equivalence over a domain as follows.
Let $f, g: \R^n \to \R$ be functions such that $\forall \vex \in \Z^n:\, f(\vex), g(\vex) \in \Z$ and let $\DD \subseteq \Z^n$.
We say that $g$ and $f$ are \emph{equivalent} on $\DD$ if
\begin{equation} \label{eq:f_equivalence}
\forall \vex, \vey \in \DD: \, f(\vex) \geq f(\vey) \Leftrightarrow g(\vex) \geq g(\vey)  .
\end{equation}
We use this definition of equivalence to define reducibility bound as follows.
\begin{definition}[$\rho$-reducibility] \label{def:reducibility}
Let $\rho: \N \to \N$. 
We say that a function $f: \Z^n \to \Z$ that satisfies some structural property $\upsilon$ is \emph{$\rho$-reducible} if, for every $N \in \N$, there exists a function $g: [-N,N]^n \to \Z$ satisfying $\upsilon$, which is equivalent to $f$ on $[-N,N]^n$ and $g_{\gap}^{[-N,N]^{n}} \leq \rho(N)$.  If there is such a pair of functions $f$ and $g$ for a given bound $\rho$, we also say that $f$ is $\rho$-reducible to $g$.
\end{definition}

Our goal here is to revisit the famous result of Frank and Tardos that can be stated as follows.
Every linear function $f: \R^n \to \R$ is $\left(2^{n^3}N^{n^2}\right)$-reducible to a linear function $g$ that can be computed in polynomial time.
Furthermore, they have shown that without loss of generality we can assume that if $f$ is a linear function, then $N\leq (n+1)!+1$ and thus $\log N \in O(n \log n)$ (see Theorem 4.2 in \cite{FT}).  In their work, they establish a strongly polynomial time algorithm for constructing the linear function $g$, but the time complexity of their algorithm is very large and thus limit the usage of their result.  The common practice of using their work for establishing a strongly polynomial time algorithm for a given problem, is as follows. First, establish a polynomial time algorithm for the problem where the weakly polynomial term of the time complexity is given as a polynomial term of the binary encoding length of the objective function. Then, using the work of \cite{FT} to argue that this is sufficient for establishing a strongly polynomial time algorithm for the problem as there is a sufficiently large box where the objective function is equivalent to another function whose gap is an exponential function of the dimension of the problem.  As one such recent example for this strategy, we refer to the use of \cite{FT} in \cite{epstein2024efficient}.  Following the results for the linear case, we say that an algorithm for computing an equivalent function has a {\em strongly polynomial time} if its time complexity is polynomial in $n+\log N$ and its reducibility bound is single exponential in $n$ and polynomial in $N$. For such an algorithm we can indeed use the equivalent function to transform a weakly polynomial time algorithm for integer programming into a strongly polynomial time algorithm for the same problem.

Since the problem of finding an equivalent function to $f$ over $[-N,N]^n$ with a minimum gap is solved trivially by sorting the different values of $f$ over $\DD=[-N,N]^n$ and letting the vector $\vex$ that is the $i$-th smallest value of $f$ be assigned the value $g(\vex)=i$.  Clearly, the gap is at most $(2N+1)^n-1$ and this is obviously best possible if all values of $f(\DD)$ are distinct, i.e., $f$ is one-to-one.  However, this trivial solution does not help algorithm designers in using such an equivalent function as such a general function cannot be optimized or analyzed without going over all points in $\DD$.  Thus, we would like to derive bounds for subsets $\upsilon$ of the functions defined over $\DD$, such that both $f$ and its equivalent function $g$ belong to $\upsilon$.  In this sense, the result of \cite{FT} is for $\upsilon$ being the set of linear functions.   Thus, we have defined the notion of reducibility bounds over the function subset $\upsilon$ over the domain $\DD$.  This is the definition used by \cite{Eisen23} although not stated there in this form.

\paragraph{Related work following the results of \cite{FT}.} Motivated by applications in compressed sensing, the problem of finding relatively small reducibility bounds was considered in \cite{banff} where the authors claim some reducibility bounds for separable convex functions (see the discussion on the second topic in this report).  This report does not state explicit bounds but they argue that it is reasonable to consider bounds for which there exists an inefficient algorithm that computes the equivalent function as well as bounds that can be computed efficiently.  Here, we both improve as well as extend the bounds to other regimes of functions $f$ for which we would like to prove the existence of an equivalent function with a small gap and to compute such functions with a small gap efficiently.

In \cite{Eisen23},  the motivation for proving such bounds was  the possible  transformation of a weakly polynomial time algorithm for block-structured integer programs with a linear objective into a strongly polynomial time algorithm.  The resulting algorithm is the one in \cite{E+22arxiv} for linear objectives. The results of \cite{Eisen23}  are upper bounds and lower bounds on the reducibility bounds of  linear functions and of separable convex functions.  The bounds of \cite{banff} and most bounds of \cite{Eisen23} are existence results that can be computed by an algorithm that runs in exponential time.    

Our focus here is on the existence reducibility bounds that we show as well as the bounds that can be computed efficiently.  Our techniques for providing bounds that can be computed efficiently are based on applying our tool for establishing existence bounds together with the strongly polynomial time algorithm of Frank and Tardos \cite{FT}.  Since such reducibility bounds that can be computed efficiently are known for the linear case (namely, the results of \cite{FT}) we will consider these results only for the other cases of $\upsilon$.

\paragraph{Paper outlines and our results.} In Section \ref{sec:prel} we present a tool for deriving reducibility bounds for specific cases of the function subset $\upsilon$ and we summarize the results of \cite{FT} that are relevant for establishing polynomial time algorithm for constructing the equivalent function.  In Section \ref{sec:lin} we consider the case of linear functions and prove an improved existence bound smaller than $2N \cdot \left((n+3)N\right)^{n}$.  For this case of linear functions, \cite{Eisen23} proved a weaker existence upper bound on the gap of an equivalent linear function of $(4nN)^n$ and they have shown a lower bound of $N(nN)^{n-1}$ on the minimum possible such gap.  Then, in Section \ref{sec:sep}, we turn our attention to separable functions (both general separable functions, as well as separable convex functions).  These are functions $h:\DD \rightarrow \Z$, that can be written as $h(\vex)=\sum_{i\in [n]} h_i(x_i)$.  For such functions we prove an existence upper bound smaller than $(nN+3/2)^{2nN}$ on the gap of an equivalent function (both for general separable functions as well as for separable convex functions).  This bound should be compared to the upper bound of $(n^2N)^{n(2N+1)+1}$ shown in \cite{Eisen23}.  We also provide our first computable bound, but for the case of separable functions, the time complexity of the algorithm for finding the equivalent function is polynomial in $nN$ and thus it is not strongly polynomial in the side of the box. In Section \ref{sec:sep-quad} we show the generality of our approach by considering the case of separable quadratic functions for which we show an improved upper bounds whose exponents are linear in $n$ whereas the general bound due to separability is only linear in $nN$.  Here, the time complexity of the algorithm for finding an equivalent function is strongly polynomial. In Section \ref{sec:quad} we consider the case of quadratic (non-separable) functions and we derive upper bounds whose exponents are linear in $n^2$. 
 We refer to the theorems in the corresponding sections for the precise values of the bounds we show.
 
Thus, in what follows we show that there is a strongly polynomial time algorithm for constructing an equivalent function with a small gap in the cases of separable quadratic functions and of quadratic functions, while the obtained algorithm for the separable functions (convex and general) has polynomial time algorithm (for the last case, the time complexity will be polynomial in $n+N$).

\section{Preliminaries - the technical tools via linear programming} \label{sec:prel}
Next, we consider the main tool that we use for our improved bounds.  We use arguments that are standard within the linear programming literature.

We say that a feasibility linear program is {\em scalable} if for every $\vex$ that is feasible for the linear program, and for every $\alpha \geq 1$, we have that the vector $\alpha \vex$ is feasible for the linear program.  We note that if all constraints are either equality constraints with right hand side equal to zero, or inequalities of the form greater or equal with non-negative right hand side, then the linear program is scalable.  

Next, consider a feasibility linear program consisting of the constraints $A\vex \geq \veb$ where $A$ has $d$ columns (i.e., a linear program in dimension $d$) with all components of $(A,\veb)$ being integers in the interval $[-a,a]$.  Assume that the linear program is scalable, feasible, and that there is an extreme point of the polyhedron $\{ \vex \in \R^d: A\vex \geq \veb\}$.  By Cramer's rule, such an extreme point is a rational vector where each component   is a rational number with both the numerator and the denominator being integers of absolute value  at most $a^d \cdot d! \in O((da)^d)$. We observe the fact that for every $d$ we have $d! \leq (\frac{d+2}{2})^{d-1}$ to conclude that these numerators and denominators are at most $a\cdot (\frac{(d+2)a}{2})^{d-1}$.  The observation that  $d! \leq (\frac{d+2}{2})^{d-1}$ holds by the mean inequality as for every positive $i$ we have $(\frac{d+2}{2} -i) \cdot (\frac{d+2}{2} +i) \leq (\frac{d+2}{2})^2$ and we apply the inequality for value of $i$ for which $\frac{d+2}{2} -i \geq 2$ and so $\frac{d+2}{2} +i \leq d$.  Specifically, for even values of $d$, we apply the inequality for every positive integer $i$ and for odd values of $d$, we apply it for every positive half integer that is not integer $i$.  By multiplying the resulting inequality we get that $d! \leq (\frac{d+2}{2})^{d-1}$ for all $d$.  We similarly can conclude that $d! \leq (\frac{d+1}{2})^d$ for all $d$.  
In our settings and in order to compare our bounds to earlier work it will be easier to state the upper bound whose exponent is $d-1$ and not $d$, and thus we will mostly use the former one.  

Obviously, we can use the Stirling's formula for the upper bound on $d!$, that is, $$d! \leq \sqrt{2\pi d} \cdot \left( \frac de \right)^d \cdot e^{\frac {1}{12\cdot d}}$$ where $e$ stands for the base of the natural logarithm.  The claims of the theorems in the sequel are based on the bound of Stirling's formula.

Furthermore, we let $B$ be the submatrix of $A$ consisting of $d$ affine independent constraints that are satisfied as equality by one (fixed) extreme point. We consider the vector that is $|det(B)|$ times the extreme point that gives a feasible solution to the linear program where every component of it is an integer with absolute value of at most $d! \cdot a^d$.  The feasibility of the resulting integer vector is a straightforward result of the scalability of the corresponding linear program. Thus, we conclude the following.  

\begin{proposition}\label{lp:prel}
Consider a feasibility scalable linear program consisting of the constraints $A\vex \geq \veb$ and with all components of $(A,\veb)$ being integers in the interval $[-a,a]$.  Assume that the linear program is feasible, and that $\{ \vex \in \R^d: A\vex \geq \veb\}$ has an extreme point. Then there is an integer feasible solution $\vex$ for the linear program satisfying $$||\vex ||_{\infty} \leq d! \cdot a^d \leq \sqrt{2\pi d} \cdot \left( \frac de \right)^d \cdot e^{\frac {1}{12d}} \cdot a^d  .$$ 
\end{proposition}

\paragraph{The relevant results of Frank and Tardos \cite{FT}.} Our use of Proposition \ref{lp:prel} is based on the following method.  We encode a function in $\upsilon$ using a small number of decision variables, such that the value of the function in each point $[-N,N]^n$ is a linear function of these variables subject to the constraint that an integer values of the decision variables results in an integer value of the function in each point.  Then, enforcing the requirement that the new function is equivalent to $f$ is equivalent to adding constraints to the linear program.  Furthermore, the structural property of the functions in $\upsilon$ is enforced by (perhaps) adding more constraints to the linear program.  The resulting linear program  satisfies the assumption of Proposition \ref{lp:prel} and so we will use its claim for showing the existence upper bound on the gap of the equivalent function.

In order to use the results of \cite{FT} we interpret the resulting integer linear program obtained by adding the constraints that the decision variables are integers as a relaxation of the problem of finding an equivalent function of a linear function of $d$ decision variables in a box $[-\tilde{N},\tilde{N}]^d$ for some integer $\tilde{N}$ (that is a function of $n,N$).  Thus, we can use the algorithm of \cite{FT} for finding a feasible solution for this integer linear program whose time complexity is polynomial in $d+\log \tilde{N}$ and achieve an equivalent function, that is, a feasible solution for the integer linear program whose gap is upper bounded by  $2^{d^3}\tilde{N}^{d^2}$.  Recall that for this linear case, we can assume without loss of generality that $\tilde{N}\leq (d+1)! +1$ and thus the time complexity of this algorithm is polynomial in $d$.  We summarize this approach.

\begin{proposition}\label{FT-prop}
If the scalable linear program can be considered as finding an equivalent function in dimension $d$ over the box $[-\tilde{N},\tilde{N}]^d$, then there is a feasible solution for the integer linear program with gap at most $2^{d^3}\tilde{N}^{d^2}$ that can be found using an algorithm whose time complexity is polynomial in $d$.
\end{proposition}

\section{Linear functions\label{sec:lin}}
Assume that  $f: \R^n \to \R$ is given as $f(\vex) = \sum_{i=1}^n f_i x_i$ where $f_i, \forall i$ is given as an input.  We would like to compute a linear function $g(\vex) = \sum_{i=1}^n g_i x_i$ that is equivalent to $f$ and has a small gap.  Here, we treat the $g_i$ for all $i$ as decision variables, and note that if these decision variables are integers, then the resulting linear function maps integer vectors to integers.  Consider the following linear program with $d=n+1$ variables, where the first $n$ variables are the $g_i$ (for $i\in [n]$) and the last decision variable is $g_{gap}$.  We observe that for every $\vex\in [-N,N]^n$, we have that $g(\vex)=\sum_{i=1}^n g_i x_i$ is a linear function of the decision variables.

The constraints of the linear program are $$g_{gap} - [g(\vex)-g(\vey) ] \geq 0, \ \ \ \forall \vex,\vey \in  [-N,N]^n$$ and the constraints enforcing the equivalence to $f$.  This last family of constraints is stated as follows. For each pair $ \vex,\vey \in  [-N,N]^n$, we have one of the
following three constraints.
\begin{enumerate} 
\item If $f(\vex) > f(\vey)$, we introduce the constraint $g(\vex)-g(\vey) \geq 1$.
\item If $f(\vex) = f(\vey)$, we introduce the constraint $g(\vex)-g(\vey) =0$.
\item If $f(\vex) < f(\vey)$, we introduce the constraint $g(\vey)-g(\vex) \geq 1$.
\end{enumerate}

Since all equality constraints have right hand side equal to zero, and all inequality constraints are of the form of greater or equal a non-negative term, we conclude that this linear program is scalable.  Furthermore it is feasible because $f$ itself satisfies the constraints together with $f_{gap}$, since it is integer for every integer vector so if two values of $f$ are not equal, then the absolute value of their difference is not strictly smaller than $1$. Furthermore, $g_{gap}$ is lower bounded by $0$, and thus the polyhedron defined using these linear inequalities has an extreme point.  

Let $A$ be the constraint matrix of this feasibility linear program and $\veb$ be the right hand side.  Then, $A$ has $d=n+1$ columns, the maximum absolute value of a coefficient in $(A,\veb)$ is at most $2N$ because $g_i$ is multiplied by $x_i-y_i$ and $x_i,y_i \in [-N,N]$ so their difference is at most $2N$.  So we have $a=2N$. By Proposition \ref{lp:prel}, we conclude that there is an integer solution with $g_{gap} \leq 2N \cdot ((n+3)N)^{n}$.  Recall that for this case \cite{Eisen23} have shown a weaker upper bound on the gap of an equivalent linear function of $(4nN)^n$ and they have shown a lower bound of $N(nN)^{n-1}$ on the gap of such equivalent function.  Here, we state the bound we can achieve by using the Sterling's formula for upper bounding $d!$.  That is, we conclude that $g_{gap} \leq d! a^d \leq \sqrt{2\pi d} \cdot \left( \frac de \right)^d \cdot e^{\frac {1}{12d}} \cdot a^d $, and using our values of $a,d$ we conclude the following.

\begin{theorem}
Let $$\rho_{lin} = \sqrt{2\pi (n+1)} \cdot \left( \frac {2N(n+1)} e \right)^{n+1} \cdot e^{\frac {1}{12(n+1)}} \ . $$
Every linear function $f: \R^n \to \R$ is $\rho_{lin} $-reducible to some linear function $g$.
\end{theorem}

\section{Separable functions\label{sec:sep}}
Next, we turn our attention to separable functions.  These are functions $f(\vex)$ that can be written as $f(\vex)=\sum_{i=1}^n f_i(x_i)$ where for all $i\in [n]$, $f_i:\Z \rightarrow \Z$ for which we require the equivalent function $g$ to be separable as well.  We will also consider the case where the separable functions need to be convex, i.e., where for all $i$ both $f_i$ and $g_i$ need to be convex, as well as the case where this is required only for some components.

The method to handle such functions is to increase the dimension of the corresponding scalable linear program.  Without loss of generality we assume that $f_i(-N)=0$ for all $i\in [n]$, and we will set $g_i(-N)=0$ for all $i\in [n]$.  Our decision variables are the increment of $g_i$ for all $i$.  That is, we will let $h_{i,k}=g_i(k)-g_i(k-1)$ be defined for $i\in [n], k\in [-N+1,N]$.  These are $2nN$ decision variables, and we add to that collection of variables another decision variable $g_{gap}$ denoting the maximum gap.  We treat the family of $h_{i,k}$ for all $i,k$ as decision variables and note that if these decision variables are integers, then defining $g_i(k')=\sum_{k=-N+1}^{k'} h_{i,k}$ for $k'\geq -N+1$ and $g_i(-N)=0$ gives an integer value for all $i,k'$.  Furthermore, $g_i(k')$ is a linear combination of the decision variables with binary coefficients.  Furthermore, letting $g(\vex)=\sum_{i\in [n]} g_i(x_i)$ is also a  linear combination of the decision variables with binary coefficients for all $\vex \in [-N,N]^n$. 

For separable functions, we use exactly the same set of linear inequalities as we did for linear functions. That is,  $g_{gap} - [g(\vex)-g(\vey) ] \geq 0, \ \ \ \forall \vex,\vey \in  [-N,N]^n$ and the constraints enforcing the equivalence to $f$.  That is, for each pair $ \vex,\vey \in  [-N,N]^n$, we have one of the
following three constraints.
\begin{enumerate} 
\item If $f(\vex) > f(\vey)$, we introduce the constraint $g(\vex)-g(\vey) \geq 1$.
\item If $f(\vex) = f(\vey)$, we introduce the constraint $g(\vex)-g(\vey) =0$.
\item If $f(\vex) < f(\vey)$, we introduce the constraint $g(\vey)-g(\vex) \geq 1$.
\end{enumerate}

If we require that $g_i$ is convex for some $i\in [n]$ (and we know that $f_i$ is convex as well), then we add the constraints of monotone non decreasing increments, that is, $h_{i,k} \leq h_{i,k+1}$ for all $k\in [-N+1,N-1]$.  Similarly, if we require that $g_i$ is concave, we add the constraints that  $h_{i,k} \geq h_{i,k+1}$ for all $k\in [-N+1,N-1]$.  Observe that these additional constraints are scalable as well (as the right hand side is zero), and that the maximum coefficient in such a row is $1$.  For all these cases, the proof continues as follows.

Similarly to the case of linear functions, we conclude that the new linear program is scalable.  Furthermore it is feasible because $f$ itself together with its implied increments satisfies the constraints together with $f_{gap}$, since it is integer for every integer vector (so if two values of $f$ are not equal, then the absolute value of their difference is not strictly smaller than $1$). Furthermore, $g_{gap}$ is lower bounded by $0$, and thus the polyhedron defined using these linear inequalities has an extreme point.  

Let $A$ be the constraint matrix of this feasibility linear program and $\veb$ be the right hand side.  Then, $A$ has $d=2Nn+1$ columns, the maximum absolute value of a coefficient in $(A,\veb)$ is $1$ because $g(\vex)$ and $g(\vey)$ are summations of subsets of the decision variables so their difference is a linear combination of the decision variables with coefficients in  $\{ -1,0,1\}$.  So we have $a=1$. By Proposition \ref{lp:prel}, we conclude that there is an integer solution with $g_{gap} \leq d! $.  In order to compare our result to the previous work, we use that $d! \leq (\frac{d+2}{2})^{d-1}$ and conclude that $g_{gap} \leq (nN+3/2)^{2nN}$. The same upper bound holds also for the case where $f,g$ are separable convex (or separable concave) and for a mixture of these cases. 
Recall that in \cite{Eisen23}, the authors considered only the case of separable convex functions for which they showed that such functions are $(n^2N)^{n(2N+1)+1}$-reducible to a separable convex function.  It is straightforward to see that our new bound is significantly better than the one from \cite{Eisen23} as both $nN+3/2 < n^2N$ and $2nN < n(2N+1)+1$. 

However, when we use the upper bound due to the Stirling's formula (in all these cases), we get the following result.

\begin{theorem}
Let $$\rho_{sep} =   \sqrt{2\pi (2nN+1)} \cdot \left( \frac {2nN+1}e \right)^{2nN+1} \cdot e^{\frac {1}{24nN+12}} \ . $$
Every separable function $f: \R^n \to \R$ is $\rho_{sep}$-reducible to some separable function $g$.  Every separable convex function $f: \R^n \to \R$ is $\rho_{sep}$-reducible to some separable convex function $g$. Every separable concave function $f: \R^n \to \R$ is $\rho_{sep}$-reducible to some separable concave function $g$.
\end{theorem}

Next, we turn our attention to equivalent functions that can be constructed using an efficient algorithm.  We establish the following result. 
\begin{theorem}
There is an algorithm whose time complexity is polynomial in $nN$ that given a separable function $f$ (separable convex function) finds an equivalent separable function $g$ (an equivalent separable convex function $g$, respectively) with gap at most $O(2^{(2nN)^3})$. 
\end{theorem}
\begin{proof}
We treat the set of variables $\{ h_{i,k} \}_{i,k}$ as a family of decision variables and define the linear function $f'$ of these $d'=2nN$ decision variables by defining the value of the function of the corresponding $d'$ unit vectors as $f'(\vece^{i,k}) =  f_i(k)-f_i(k-1)$ for all $i,k$ where $\vece^{i,k}$ is the unit vector with $1$ in the component corresponding to $h_{i,k}$.  

Then, a new function $g'$ that is equivalent to $f'$ over the unit cube in dimension $2nN$ naturally defines an equivalent function $g$ to $f$ in dimension $n$.  This is carried out by letting $g_i(k')=\sum_{k=-N+1}^{k'} g'(\vece^{i,k})$ for all $i,k'$ and $g(\vex)=\sum_{i=1}^n g_i(x_i)$ for all $\vex$.  

Thus, it suffices to construct an equivalent function to $f'$ (in dimension $d'=2nN$) over the unit cube so $\tilde{N}=1$.  The resulting gap of $g$ is not larger than the gap of $g'$ over the unit cube, so by Proposition \ref{FT-prop} the gap of $g$ is at most $2^{d^3}\tilde{N}^{d^2}=2^{(2nN)^3}$.  Since we construct $g$ as a separable function, we conclude that it is indeed separable.  

The requirement that $g_i$ is convex (concave) for a given $i$ is implied by the equivalence to $f'$ and the fact that the $i$-th coordinate of $f'$ is convex (concave, respectively).  This is so for $f_i$ convex as requiring that $h_{i,k}\geq h_{i,k-1}$ is the same as requiring that $g(\vece^{i,k}) \geq g(\vece^{i,k-1})$ and this is implied by $f'(\vece^{i,k}) \geq f'(\vece^{i,k-1})$ that holds as $f_i$ is convex (and similarly for $f_i$ concave).
\end{proof}

With respect to the constructive bound for the separable convex function due to \cite{Eisen23} (i.e., Theorem 3 part 2 in \cite{Eisen23}), we note the following.  In that work the approach of using a constructive bound for the linear case is also considered.  But in their application $d=n(2N+1)$ is larger than our dimension, and the box edge length $\tilde{N}=n$ is significantly larger than our constant box edge length $\tilde{N}=1$.  Thus, our approach leads to a smaller bound than the one of \cite{Eisen23}.

\section{Separable quadratic functions\label{sec:sep-quad}}
Next, we turn our attention to separable quadratic functions.  These are special type of separable functions for which we are able to guarantee smaller bounds.  These are functions $f(\vex)$ that can be written as $f(\vex)=\sum_{i=1}^n f_i(x_i)$ for which $f_i(x_i)$ can be written as $f_i(x_i)=\alpha^f_i \cdot x_i^2+\beta^f_i \cdot x_i +\gamma_i$ where $\alpha^f_i, \beta^f_i, \gamma_i \in \R$ and we require the equivalent function $g$ to be of the same form.  

The method to handle such functions is to modify our solution for linear functions.  Without loss of generality we assume that $f_i(0)=0$ for all $i\in [n]$, and we will set $g_i(0)=0$ for all $i\in [n]$. 
Thus, we will assume that for all $i\in [n]$, $g_i(x_i) = \alpha^g_i \cdot x_i^2+\beta^g_i \cdot x_i$.
 Our decision variables are the multipliers $\alpha^g_i, \beta^g_i$ for all $i$.  These are $2n$ decision variables, and we add to that another decision variable $g_{gap}$ denoting the maximum gap.  We note that if these decision variables are integers, then defining $g_i(x_i)=\alpha^g_i \cdot x_i^2+\beta^g_i \cdot x_i$  gives an integer value for all $i\in [n],x_i \in [-N,N]$.  Furthermore, $g(\vex)$ is a linear combination of the decision variables with integer coefficients with a maximum absolute value of the coefficients of at most $N^2$. 

We use exactly the same set of linear inequalities as we did for linear functions. That is,  $g_{gap} - [g(\vex)-g(\vey) ] \geq 0, \ \ \ \forall \vex,\vey \in  [-N,N]^n$ and the constraints enforcing the equivalence to $f$.  That is, for each pair $ \vex,\vey \in  [-N,N]^n$, we have one of the
following three constraints.
\begin{enumerate} 
\item If $f(\vex) > f(\vey)$, we introduce the constraint $g(\vex)-g(\vey) \geq 1$.
\item If $f(\vex) = f(\vey)$, we introduce the constraint $g(\vex)-g(\vey) =0$.
\item If $f(\vex) < f(\vey)$, we introduce the constraint $g(\vey)-g(\vex) \geq 1$.
\end{enumerate}

Similarly to the case of linear functions, we conclude that the new linear program is scalable.  Furthermore it is feasible because $f$ itself together with its implied parameters satisfies the constraints together with $f_{gap}$, since it is integer for every integer vector. Furthermore, $g_{gap}$ is lower bounded by $0$, and thus the polyhedron defined using these linear inequalities has an extreme point.  
Let $A$ be the constraint matrix of this feasibility linear program and $\veb$ be the right hand side.  Then, $A$ has $d=2n+1$ columns, the maximum absolute value of a coefficient in $(A,\veb)$ is $a=2N^2$.   Using Proposition \ref{lp:prel}, we conclude that there is an integer solution with $g_{gap} \leq d! a^d \leq \sqrt{2\pi d} \cdot \left( \frac de \right)^d \cdot e^{\frac {1}{12d}} \cdot a^d = \sqrt{2\pi (2n+1)} \cdot \left( \frac {(2n+1)\cdot 2N^2}{e} \right)^{2n+1} \cdot e^{\frac {1}{12(2n+1)}}$. The same upper bound holds also for the case where $f,g$ are separable quadratic convex functions for which we require $\alpha^g_i \geq 0$ for all $i\in [n]$ as additional constraints in the scalable linear program.  These additional constraints do not change the parameters of $a,d$ and we get the same bounds. We thus obtained the following result.

\begin{theorem}
Let $$\rho_{sep-quad}= \sqrt{2\pi (2n+1)} \cdot \left( \frac {(2n+1)\cdot 2N^2}{e} \right)^{2n+1} \cdot e^{\frac {1}{12(2n+1)}} . $$ 
Every separable quadratic function $f: \R^n \to \R$ is $\rho_{sep-quad}$-reducible to some separable quadratic function $g$.  Every separable convex quadratic function $f: \R^n \to \R$ is $\rho_{sep-quad}$-reducible to some separable convex quadratic function $g$.
\end{theorem}

Next, we turn our attention to the possibility of constructing an equivalent function of a small gap in polynomial time.   We establish the following result.

\begin{theorem}
There is an algorithm whose time complexity is polynomial in $n$ that given a separable quadratic function $f$ finds an equivalent separable quadratic function $g$ with gap at most $2^{8n^3} N^{8n^2}$. 
\end{theorem}
\begin{proof}
We consider the function $f'$ defined over the box $[-N^2,N^2]^{2n}$ where the function $f'$ is the linear function with coefficients $\alpha^f_i,\beta^f_i$ in components $2i-1$ and $2i$ respectively.  We are interested in finding an equivalent function over a subset of the box, that is, points $\vey$ where for every $i$, $y_{2i-1}=y_{2i}^2$.  However, $f'$ is defined over all points in the box and over the integer points, it is integer.  Thus, we have a formulation where the dimension is $d=2n$ and the value of $\tilde{N}$ is $N^2$.  An equivalent linear function $g'$ to $f'$ defines integer coefficients $\alpha^{g'}_i,\beta^{g'}_i$ for all $i$, and using these coefficients we define the separable quadratic function  $$g(\vex) = \sum_{i\in [n]}\left( \alpha^{g'}_i \cdot x_i^2+\beta^{g'}_i \cdot x_i \right) . $$  

The gap of this separable quadratic function $g$ over $[-N,N]^n$ is at most the gap of the linear function $g'$ over $[-N^2,N^2]^{2n}$.  Using Proposition \ref{FT-prop}, this gap is at most $2^{(2n)^3} N^{2\cdot (2n)^2} = 2^{8n^3} N^{8n^2}$ and a linear function satisfying these conditions can be derived using an algorithm whose time complexity is polynomial in $n$.  

The function $g$ is indeed equivalent to the function $f$ as for two points $\vex,\vey \in [-N,N]^n$ for which the two corresponding points $\vex',\vey'\in [-N^2,N^2]^{2n}$ defined as $x'_{2i-1} = x_i^2, x'_{2i}=x_i, y'_{2i-1}=y_i^2, y'_{2i}=y_i$ for all $i\in [n]$, we have that $f(\vex)\geq f(\vey)$ if and only if $f'(\vex')\geq f'(\vey')$ and similarly $g(\vex)\geq g(\vey)$ if and only if $g'(\vex')\geq g'(\vey')$.  Thus, the equivalence between $f$ and $g$ follows from the equivalence between $f'$ and $g'$. 
\end{proof}

\section{Non-separable quadratic functions\label{sec:quad}}
Last, in this section we consider quadratic functions (not necessarily separable functions).  These are  functions $f(\vex)$ that can be written as $f(\vex)=\sum_{i\in [n]} \sum_{j=i}^n \alpha^f_{ij} \cdot x_i \cdot x_j +\sum_{i\in [n]} \beta^f_i \cdot x_i +\gamma^f$ and we require the equivalent function $g$ to be of the same form.  

Once again we use the same method.  Without loss of generality we assume that $\gamma^f=\gamma^g=0$. 
 Our decision variables are the multipliers $\alpha^g_{ij}, \beta^g_i$ for all $i,j$.  These are $\frac{n(n+3)}{2}$ decision variables, and we add to that another decision variable $g_{gap}$ denoting the maximum gap.  We treat this family of decision variables and note that if these decision variables are integers, then defining $g(\vex)=\sum_{i\in [n]} \sum_{j=i}^n \alpha^g_{ij} \cdot x_i \cdot x_j +\sum_{i\in [n]} \beta^g_i \cdot x_i  $ gives an integer value for all $\vex$.  Furthermore, $g(\vex)$ is a linear combination of the decision variables with integer coefficients with a maximum absolute value of the coefficients of at most $N^2$. 

We use exactly the same set of linear inequalities as we did for linear functions. That is,  $g_{gap} - [g(\vex)-g(\vey) ] \geq 0, \ \ \ \forall \vex,\vey \in  [-N,N]^n$ and the constraints enforcing the equivalence to $f$.  That is, for each pair $ \vex,\vey \in  [-N,N]^n$, we have one of the
following three constraints.
\begin{enumerate} 
\item If $f(\vex) > f(\vey)$, we introduce the constraint $g(\vex)-g(\vey) \geq 1$.
\item If $f(\vex) = f(\vey)$, we introduce the constraint $g(\vex)-g(\vey) =0$.
\item If $f(\vex) < f(\vey)$, we introduce the constraint $g(\vey)-g(\vex) \geq 1$.
\end{enumerate}

Similarly to the case of linear functions, we conclude that the new linear program is scalable.  Furthermore it is feasible because $f$ itself with its implied parameters satisfies the constraints together with $f_{gap}$, since it is integer for every integer vector. Furthermore, $g_{gap}$ is lower bounded by $0$, and thus the polyhedron defined using these linear inequalities has an extreme point.  
Let $A$ be the constraint matrix of this feasibility linear program and $\veb$ be the right hand side.  Then, $A$ has $d=\frac{n(n+3)}{2}+1$ columns, the maximum absolute value of a coefficient in $(A,\veb)$ is $a=2N^2$.   By Proposition \ref{lp:prel}, we conclude that there is an integer solution with $g_{gap} \leq \sqrt{2\pi d} \cdot \left( \frac de \right)^d \cdot e^{\frac {1}{12d}} \cdot a^d = \sqrt{2\pi \cdot( \frac{n(n+3)}{2}+1)} \cdot \left( \frac {(\frac{n(n+3)}{2}+1) \cdot (2N^2)}{e} \right)^{\frac{n(n+3)}{2}+1} \cdot e^{\frac {1}{12\cdot(\frac{n(n+3)}{2}+1)}}$.  We thus obtained the following result.


\begin{theorem}
Let $$\rho_{quad}=\sqrt{2\pi \cdot( \frac{n(n+3)}{2}+1)} \cdot \left( \frac {(\frac{n(n+3)}{2}+1) \cdot (2N^2)}{e} \right)^{\frac{n(n+3)}{2}+1} \cdot e^{\frac{1}{6n(n+3)+12 }} .$$
Every quadratic function $f: \R^n \to \R$ is $\rho_{quad}$-reducible to some quadratic function $g$. 
\end{theorem}

We conclude by considering the possibility of constructing an equivalent function of a small gap in polynomial time also for this case.   We establish the following result.

\begin{theorem}
There is an algorithm whose time complexity is polynomial in $n$ that given a quadratic function $f$ finds an equivalent quadratic function $g$ with gap at most $$2^{(\frac{n(n+3)}{2})^3} N^{2(\frac{n(n+3)}{2})^2} . $$
\end{theorem}
\begin{proof}
We consider the function $f'$ defined over the box $[-N^2,N^2]^{\frac{n(n+3)}{2}}$ where the function $f'$ is the linear function with coefficients $\alpha^f_{ij}$ for all $i,j\in [n]$ such that $i\leq j$ and $\beta^f_i$ for $i\in [n]$ where the first $n$ components are $\beta^f_i$ for $i\in [n]$ and the other components are the $\alpha^f_{ij}$.  We let $\pi(k)$ be the two indexes $ij$ for which $\alpha^f_{ij}$ is written in the $k$-th component of such vector and let $\pi^{-1}(ij) =k$ be the inverse of this function (so that $\pi(k)=ij$.  We are interested in finding an equivalent function over a subset of the box, that is, points $\vey$ where for every $i,j\in [n]$ such that $i\leq j$, we have $y_i \cdot y_j =y_{\pi^{-1}(ij)}$.  However, $f'$ is defined over all points in the box and over the integer points, it is integer.  Thus, we have a formulation where the dimension is $d=\frac{n(n+3)}{2}$ and the value of $\tilde{N}$ is $N^2$.  An equivalent linear function $g'$ to $f'$ defines integer coefficients $\alpha^{g'}_{ij},\beta^{g'}_i$ for all $i,j$, and using these coefficients we define the quadratic function   $$g(\vex)=\sum_{i\in [n]} \sum_{j=i}^n \alpha^{g'}_{ij} \cdot x_i \cdot x_j +\sum_{i\in [n]} \beta^{g'}_i \cdot x_i   . $$

The gap of this quadratic function $g$ over $[-N,N]^n$ is at most the gap of the linear function $g'$ over $[-N^2,N^2]^{\frac{n(n+3)}{2}}$.  Using Proposition \ref{FT-prop}, this gap is at most $2^{(\frac{n(n+3)}{2})^3} N^{2(\frac{n(n+3)}{2})^2}$ and a linear function satisfying these conditions can be derived using an algorithm whose time complexity is polynomial in $n$.  

For a point $\vex\in [-N,N]^n$ we define a corresponding point $\vex'\in [-N^2,N^2]^{\frac{n(n+3)}{2}}$ where the first $n$ components of $\vex'$ are identical to $\vex$ (i.e., $x'_i=x_i$ for $i\in [n]$), and for $k>n$ for which $\pi^{-1}(k)=ij$, the $k$-th component of $\vex'$ is $x'_k = x_i \cdot x_j$
The function $g$ is indeed equivalent to the function $f$ as for two points $\vex,\vey \in [-N,N]^n$ for which the two corresponding points $\vex',\vey'\in [-N^2,N^2]^{\frac{n(n+3)}{2}}$, we have that $f(\vex)\geq f(\vey)$ if and only if $f'(\vex')\geq f'(\vey')$ and similarly $g(\vex)\geq g(\vey)$ if and only if $g'(\vex')\geq g'(\vey')$.  Thus, the equivalence between $f$ and $g$ follows from the equivalence between $f'$ and $g'$. 
\end{proof}

\section*{Acknowledgments}
The author is partially supported by ISF -  Israel Science Foundation grant number 1467/22.

\bibliographystyle{plainurl}


\end{document}